\documentclass{imsart}

\RequirePackage{natbib}
\RequirePackage[colorlinks,citecolor=blue,urlcolor=blue]{hyperref}

\RequirePackage{subcaption}
\RequirePackage{graphicx}
\usepackage{amsmath, amsfonts, amssymb, amsthm, graphicx}
\usepackage{verbatim}

\theoremstyle{plain}
\newtheorem{theorem}{Theorem}
\newtheorem{corollary}{Corollary}

\startlocaldefs

\newcommand{\norm}[1]{\left\lVert#1\right\rVert}
\newcommand{\B}{\mathcal{B}}
\newcommand{\W}{\mathcal{W}}
\newcommand{\e}{\epsilon}

\renewcommand{\l}{\ell}

\newcommand{\E}{\mathbb{E}}

\newcommand{\R}{\mathbb{R}}

\newcommand{\A}{\mathcal{A}}
\newcommand{\V}{\mathcal{V}}
\newcommand{\tr}{\text{tr}}
\newcommand{\vectorize}{\text{vec}}

\newcommand\mat[1]{\begin{pmatrix}#1\end{pmatrix}} 

\let\temp\phi
\let\phi\varphi
\let\varphi\temp

\endlocaldefs

\begin{document}

\begin{frontmatter}

\title{Geometric ergodicity of Gibbs samplers for Bayesian error-in-variable regression}
\runtitle{Geometric ergodicity of Gibbs samplers for Bayesian EIV regression}

\author{\fnms{Austin} \snm{Brown}\corref{}\ead[label=e1]{austin.d.brown@warwick.ac.uk}}
\address{Department of Statistics, University of Warwick, Coventry, UK \printead{e1}}

\runauthor{A. Brown}

\begin{abstract}
Multivariate Bayesian error-in-variable (EIV) linear regression is considered to account for additional additive Gaussian error in the features and response. A 3-variable deterministic scan Gibbs samplers is constructed for multivariate EIV regression models using classical and Berkson errors with independent normal and inverse-Wishart priors. These Gibbs samplers are proven to always be geometrically ergodic which ensures a central limit theorem for many time averages from the Markov chains. We demonstrate the strengths and limitations of the Gibbs sampler with simulated data for large data problems, robustness to misspecification and also analyze a real-data example in astrophysics.
\end{abstract}



\received{\smonth{4} \syear{2023}}

\tableofcontents

\end{frontmatter}

\section{Introduction}

Many problems in astrophysics \citep{Feigelson1992, Hilbe2017, Kelly2012, Stefanski2000} and epidemiology \citep{Achic2018, Buonaccorsi2010, Carroll2006, Clayton1992} among other areas of science \citep{Marcus2016, Pollice2019, Tang2017} involve error in variables (EIV) which classical linear regression does not take into account.
EIV can occur in many situations such as measurement error in data collection \citep{Hilbe2017, Kelly2012}, discrepancies between the data distribution and the model \citep{Carroll2006, Buonaccorsi2010}, or purposeful adversarial attacks against the data \citep{Goodfellow2015, Szegedy2014}.
Not surprisingly, multiple critical issues arise in parameter estimation and statistical inference when ignoring additional errors in the data such as poor predictive performance \citep{Goodfellow2015}, statistical bias \citep{Damgaard2020, Kroger2016, Vidal2008}, and estimators fail to be consistent \citep{Michalek1980}.

Bayesian approaches develop a strategy for additional error in the variables by constructing a new model incorporating additional error.
We consider multivariate Bayesian EIV linear regression \citep{Farr2020, Dellaportas1995, Xing2017, Huang2010, Mallick1996, Muff2015, Richardson1993, Rodrigues2007, Torabi2021, Vidal2010} accounting for additive Gaussian error in the features (covariates) and response.
We assume the variability of the additive Gaussian error is known beforehand which arises often in astrophysics in the presence of known instrumentation error \citep{Hilbe2017, Kelly2012}.
Alternative approaches to EIV models attempt to correct existing parameter estimation methods such as least squares or method of moments with weighting and other techniques \citep{Fuller1987, Stefanski1985}.
Several other strategies for EIV modeling are discussed in more comprehensive treatments on the topic \citep{Buonaccorsi2010, Carroll2006, Fuller1987}.


We write $x \sim N_d(m, C)$ to mean the $d$-dimensional normal distribution with mean $m$ and symmetric, positive-definite (SPD) covariance matrix $C$.
We also write $x \sim \W_d^{-1}(\nu, V)$ to be the inverse-Wishart distribution with positive integer degrees of freedom $\nu \ge d$ and scale SPD matrix $V \in \R^{d \times d}$.
Let $\vectorize(A)$ denote the vectorization of a matrix $A$ by stacking the columns. 
Let $(Y_i, X_i, Z_i)_{i = 1}^n$ be independent and identically distributed (i.i.d.) where the response $Y_i$ takes values in $\R^m$ along with features $X_i$ taking values in $\R^p$ and fixed, known features $Z_i \in \R^q$ where $m, n, p, q$ are positive integers.
Let $\theta = \vectorize(\Theta) \in \R^{qm}$, $\beta = \vectorize(\B) \in \R^{pm}$, and SPD matrix $\Sigma \in \R^{m \times m}$ be unknown regression and covariance parameters respectively.
We introduce new parameters $\A = (\A_1, \ldots, \A_n)^T$ with $\A_i \in \R^p$ to model additional error in $X_i$ using classical or Berkson errors \citep{Berkson1950}. 
The classical error model specifies $X_i | \A_i$ and the Berkson error model \citep{Berkson1950} assumes instead a data-dependent prior on $\A_i | X_i$.
When there is additional error in $X_i$, the EIV linear regression model for $i \in 1, \ldots, n$ is i.i.d. with
\begin{subequations}
\label{eq:EIV_regression_model_X}
\begin{align}
&Y_i | \A_i, \theta, \beta, \Sigma \sim N_m( \Theta^T Z_i + \B^T \A_i, \Sigma)
\\
& X_i | \A_i \sim N_p(\A_i, V_{i}) \text{ (Classical)}
\text{\hspace{.4cm} or \hspace{.4cm}} 
\A_i | X_i \sim N_p(X_i, V_{i}) \text{ (Berkson)}
\end{align}
\end{subequations}
where the SPD matrices $V_i \in \R^{p \times p}$ are known.
When there is also additional error in the responses $Y_i$, we assume an i.i.d. hierarchical regression model with
\begin{subequations}
\label{eq:EIV_regression_model_X_Y}
\begin{align}
&Y_i | \V_i \sim N_m( \V_i, U_{i})
\\
&\V_i | \A_i, \theta, \beta, \Sigma \sim N_m( \Theta^T Z_i + \B^T \A_i, \Sigma)
\\
& X_i | \A_i \sim N_p(\A_i, V_{i}) \text{ (Classical)}
\text{\hspace{.4cm} or \hspace{.4cm}} 
\A_i | X_i \sim N_p(X_i, V_{i}) \text{ (Berkson)}
\end{align}
\end{subequations}
where $U_{i} \in \R^{m \times m}$ are known SPD matrices.

We will be interested in the posterior for both models \eqref{eq:EIV_regression_model_X} and \eqref{eq:EIV_regression_model_X_Y} using independent normal and inverse-Wishart priors on the parameters $(\A, \theta, \beta, \Sigma)$. 
The independent prior choice is a popular choice in Bayesian regression models with and without measurement error \citep{Carroll2006, Dellaportas1995, KarlOskar2019, raja:spar:2015}.
For the EIV regression models \eqref{eq:EIV_regression_model_X} and \eqref{eq:EIV_regression_model_X_Y}, the independent priors are chosen
\begin{subequations}
\label{eq:EIV_priors}
\begin{align}
&\Sigma \sim \W^{-1}(a_0, B_0)
\\
&(\theta, \beta)^T \sim N_{m(q + p)}(j_0, J_0)
\end{align}
\end{subequations}
where $a_0 \ge m$ is a positive integer, $B_0 \in \R^{m \times m}$ is a SPD matrix, $j_0 \in \R^{m(q + p)}$ and SPD matrix $J_0 \in \R^{m(q + p) \times m(q + p)}$.
The classical and Berkson error models assume either
\begin{align}
\label{eq:EIV_priors_A}
& \A_i \sim N_{p}(k_i, K_i) \text{ (Classical)}
\text{\hspace{.4cm} or \hspace{.4cm}} \A_i \text{ flat prior} \text{ (Berkson)}
\end{align}
where $k_i \in \R^p$ and $K_i \in \R^{p \times p}$ are SPD matrices.
For example, an \textit{exposure model} \citep{Gustafson2003} utilized often in epidemiology would assume classical errors with a data-dependent prior on each $\A_i$ depending on $Z_i$.
In the Berkson error model, each $\A_i | X_i$ is already specified and it is natural to assume an improper flat prior on each $\A_i$.

Previous work has proposed Gibbs sampling \citep{Geman1984} to draw samples from the posterior, denoted by $\Pi_n$, in Bayesian EIV regression models \citep{Bhadra2016, Carroll2006, Dellaportas1995, Richardson1993}.
However, trustworthy estimation from a Gibbs sampler requires the Markov chain to converge to the posterior distribution at a sufficiently fast rate.
Consider a vector-valued function $f$ with $\int \norm{f}^{2 + \delta} d\Pi_{n} < \infty$ for some $\delta \in (0, \infty)$ and denote $\bar{f}_{m}$ as the time average of $m$ samples from the Gibbs sampler.
In order to be confident in the estimator $\bar{f}_{m}$ in applications, a standard error and confidence interval are essential.
A Gibbs sampler is geometrically ergodic if initialized at points, its marginal distribution is converging to $\Pi_n$ at an exponential rate in total variation.
Geometrically ergodic Gibbs samplers provide rich theoretical guarantees which are of practical relevance in applications.
These Gibbs samplers satisfy a central limit theorem \citep{Chan1994, Jones2004}, that is,
\[
\sqrt{m} \left( \bar{f}_{m} - \int f d\Pi_n \right)
\]
is asymptotically normally distributed and under suitable assumptions, the covariance in this normal distribution can be consistently estimated \citep{Flegal2010}.
Further pertinent tools to ensuring reliable estimation such as estimates of the effective sample size, consistent confidence ellipsoids, and consistent confidence intervals for quantile estimation are also available \citep{Doss2014, Vats2019}.

To the best of our knowledge, the rate of convergence for Gibbs sampling in EIV regression models has not been previously investigated.
Related approaches have instead proposed variational Bayesian methods \citep{Bresson2021, Pham2013} and the integrated nested Laplace approximation (INLA) \citep{Rue2009, Muff2015}.
We construct a general density which in special cases, is the posterior for the $4$ Bayesian EIV regression models \eqref{eq:EIV_regression_model_X} and \eqref{eq:EIV_regression_model_X_Y} using the independent normal and inverse-Wishart prior choice on the parameters \eqref{eq:EIV_priors} and \eqref{eq:EIV_priors_A}.
Our main contribution constructs a 3-variable deterministic scan Gibbs sampler for this general density, and we show it is \textit{always} geometrically ergodic using a drift and minorization condition \citep{Hairer2011, Meyn2009}.
The 3-variable Gibbs sampler we construct can be simulated efficiently on a computer without the need for complex Metropolis-Hastings or rejection sampling steps at each iteration.

The organization of this paper is as follows.
In Section~\ref{section:gibbs_sampler}, we construct a general EIV regression density and construct a 3-variable Gibbs sampler for this density.
We show the Gibbs sampler is always geometrically ergodic and apply this to the $4$ multivariate Bayesian EIV regression models presented in this introduction.
Section~\ref{section:simulation} studies the algorithm empirically where we demonstrate limitations of the Gibbs sampler with simulated data for large data problems and also the behavior of the Gibbs sampler under model misspecification. 
Section~\ref{section:astro_example} studies a real-data example in astrophysics to study supermassive black hole mass \citep{Harris2014, Hilbe2017}.
Finally in Section~\ref{section:conclusion_me}, we discuss our results and future research directions.

\section{General Gibbs Sampler for EIV regression}
\label{section:gibbs_sampler}

For positive integers $p$, define $p$-norms by $\norm{\cdot}_p$ and the Frobenius norm by $\norm{\cdot}_F$. 
Let $\otimes$ denote the Kronecker product.
The posteriors for the Bayesian EIV regression models \eqref{eq:EIV_regression_model_X} and \eqref{eq:EIV_regression_model_X_Y} using independent prior choices \eqref{eq:EIV_priors} and \eqref{eq:EIV_priors_A} for both classical and Berkson errors share a common general form which we study in this section.
The posterior densities for these Bayesian EIV regression models are special cases of the density \eqref{eq:general_posterior} but will differ depending on the EIV modeling choice illustrated in the subsequent sections.
For $i \in 1, \ldots, n$, define hyper-parameters $a_0 \in (0, \infty)$, SPD matrix $B_0 \in \R^{m \times m}$, $c_0 \in \R^{m(p + q)}$, SPD matrices $C_{0} \in \R^{m(p + q) \times m(p + q)}$, $D_i \in \R^{p \times p}$, and $d_i \in \R^p$, $R = (R_1, \ldots, R_n)^T \in \R^{n \times m}$, and $M \in \R^{n \times q}$.
For $\A \in \R^{n \times p}$, $\theta = \vectorize(\Theta) \in \R^{mq}$, $\beta = \vectorize(\B) \in \R^{mp}$, SPD matrix $\Sigma \in \R^{m \times m}$, define the density
\begin{subequations}
\label{eq:general_posterior}
\begin{align}
&\pi_n(\A, \theta, \beta, \Sigma)
\\
&\propto
\left( \frac{1}{\det(\Sigma)} \right)^{(n + a_0 + 1 + m)/2}
\exp\left[
- \frac{1}{2} \tr( \Sigma^{-1} B_0 )
\right]
\\
&\hspace{.4cm} 
\times
\exp\left[
-\frac{1}{2} \sum_{i = 1}^n ( R_i - \Theta^T M_i - \B^T \A_i )^T \Sigma^{-1} ( R_i - \Theta^T M_i - \B^T \A_i )
\right]
\\
&\hspace{.4cm} 
\times 
\exp\left( - \frac{1}{2} \sum_{i = 1}^n (\A_i - d_i)^T  D_i^{-1} (\A_i - d_i)
\right)
\\
&\hspace{.4cm} 
\times \exp\left( - \frac{1}{2} ((\theta, \beta)^T - c_0)^T  C_0^{-1} ((\theta, \beta)^T - c_0) \right).
\end{align}
\end{subequations}

We will construct a 3-variable deterministic scan Gibbs sampler using a specific update order for the density \eqref{eq:general_posterior}.
We also derive the conditional densities for the Gibbs sampler which can be sampled directly.
Initialize $(\theta_0, \beta_0, \Sigma_0)$ and $\A_0 = (\A_{1, 0}, \ldots, \A_{n, 0})$ from an initial distribution.
For $t \in 1, \ldots$, first generate 
\begin{align*}
&\Sigma_t | \A_{t-1}, \theta_{t-1}, \beta_{t-1}
\\
&\sim \W^{-1}\left( 
n + a_0, 
(R - M \Theta_{t-1} - \A_{t-1} \B_{t-1} )^T (R - M \Theta_{t-1} - \A_{t-1} \B_{t-1} )
+ B_0
\right)
\end{align*}
Next, generate
$
(\theta_t, \beta_t)^T | \A_{t-1}, \Sigma_{t}
\sim N_{m(p + q)}(
c_n(\A_{t-1}, \Sigma_{t}),
C_n(\A_{t-1}, \Sigma_{t}) 
)
$
where
\begin{align*}
&C_n(\A_{t-1}, \Sigma_t)
= \left( 
\Sigma^{-1}_t \otimes \mat{ M & \A_{t-1} }^T \mat{ M & \A_{t-1} }
+ C_{0}^{-1} 
\right)^{-1}
\\
&c_n(\A_{t-1}, \Sigma_{t})
= C_n(\A_{t-1}, \Sigma_{t}) 
\left( 
\left[ \Sigma^{-1}_t \otimes \mat{ M & \A_{t-1} }^T \right] \vectorize(R) 
+ C_0^{-1} c_0
\right).
\end{align*}
Finally, generate independently
\[
\A_{i, t} | \theta_{t}, \beta_{t}, \Sigma_{t}
\sim N_p(
d_{n, i}(\theta_{t}, \beta_{t}, \Sigma_{t}),
D_{n, i}(\theta_{t}, \beta_{t}, \Sigma_{t})
)
\]
where
\begin{align*}
&D_{n, i}(\beta_{t}, \Sigma_{t}) 
= \left( \B_{t} \Sigma^{-1}_t \B_{t}^T + D_{i}^{-1} \right)^{-1}
\\
&d_{n, i}(\theta_{t}, \beta_{t}, \Sigma_{t}) 
= D_{n, i}(\beta_{t}, \Sigma_{t}) 
\left[ D_i^{-1} d_i + \B_{t} \Sigma_{t}^{-1} \left( R_i  - \Theta_{t}^T M_i \right) \right]
\end{align*}
to obtain $\A_t = (\A_{1, t}, \ldots, \A_{n, t})^T$.

For points $(\A, \theta, \beta, \Sigma)$ and $(\A', \theta', \beta', {\Sigma}')$, the Gibbs sampler has Markov transition density 
\begin{align*}
p\left( (\A, \theta, \beta, \Sigma), (\A', \theta', \beta', {\Sigma}') \right) 
&=
\pi_{n}(\A' | \theta', \beta', {\Sigma}')
\pi_{n}(\theta', \beta' | \A, {\Sigma}')
\pi_{n}({\Sigma}' | \A, \theta, \beta)
\end{align*}
and Markov transition kernel defined for suitable sets $B$ by
\begin{align*}
P\left( (\A, \theta, \beta, \Sigma), B \right)
= \int \int \int_{B}
p\left( (\A, \theta, \beta, \Sigma), (\A', \theta', \beta', {\Sigma}') \right)
d{\A}' d{\theta'}  d{\beta}'  d{\Sigma}'. 
\end{align*}
The Markov kernel at larger iteration times $t \ge 2$ is defined recursively with $P^1 \equiv P$ by
\[
P^t\left( (\A, \theta, \beta, \Sigma), B \right) = \int P^{t-1}\left( \cdot, B \right) dP\left( (\A, \theta, \beta, \Sigma), \cdot \right).
\]
We will use the following drift function defined by
\[
V(\A, \theta, \beta)
= 
\frac{1}{2} \sum_{i = 1}^n (\A_i - d_i)^T D_{i}^{-1} (\A_i - d_i)
+ \frac{1}{2} (\theta, \beta) C_{0}^{-1} (\theta, \beta)^T
\]
combined with a minorization condition to show there is a $\rho \in (0, 1)$ and $M_0 \in (0, \infty)$ so that for any initialization $\A, \theta, \beta, \Sigma$,
\begin{align}
\sup_{|\phi| \le 1 + M_0 V} \left| \int \phi dP^t\left( (\A, \theta, \beta, \Sigma), \cdot \right) - \int \phi d\Pi_{n} \right|
\le M(\A, \theta, \beta) \rho^t
\label{eq:V_convergence}
\end{align}
where $M(\A, \theta, \beta) = 2 + M_0 V(\A, \theta, \beta) + M_0 \int V d\Pi_n$ \citep{Hairer2011}.
The condition \eqref{eq:V_convergence} implies the Gibbs sampler is geometrically ergodic.
We now state our main result.

\begin{theorem}
\label{theorem:ge}
The  3-variable deterministic scan Gibbs sampler $(\A_t, \theta_t, \beta_t, \Sigma_t)_{t = 0}^{\infty}$ for the general density \eqref{eq:general_posterior} is geometrically ergodic.
\end{theorem}
\begin{proof} 
Using a special property of the Gibbs sampler, it will be sufficient to develop a drift and minorization condition based only on the marginal chain $(\A_t, \theta_t, \beta_t)_t$ \citep[Example 3.6]{Roberts2001_denitialize}.
In particular, we will use the special property of this Gibbs Markov kernel $P$ that for suitable sets $B$, $P(\cdot, B)$ is a function of only the parameters $(\A, \theta, \beta)$ and does not depend on $\Sigma$.
We first show a minorization condition.
For $\l \in (0, \infty)$, define the function $g_r$ by 
\[
g_r(\A', \theta', \beta', {\Sigma}')
= \inf_{V(\A, \theta, \beta) \le \l} \pi_{n}(\A' | \theta', \beta', {\Sigma}')
\pi_{n}(\theta', \beta' | \A, {\Sigma}')
\pi_{n}({\Sigma}' | \A, \theta, \beta)
\]
and the constant
$
Z_{g_\l} 
= \int g_\l(\A', \theta', \beta', {\Sigma}')
d\A' d\theta' d\beta' d{\Sigma}'.
$
The drift function $V$ is a continuous, strongly convex function on a closed, convex domain so its sublevel sets are closed and bounded \citep[Corollary 3.2.2]{Nesterov2018}.
For fixed $\theta', \beta', {\Sigma}'$, the function 
\[
(\A, \theta, \beta) \mapsto \pi_n(\theta', \beta' | \A, {\Sigma}') \pi_n({\Sigma}' | \A, \theta, \beta)
\]
is continuous and achieves its minimum over compact sets.
Thus, $Z_{g_\l}$ is not $0$ and we can define the probability measure \[
\nu_\l(\cdot) = Z_{g_\l}^{-1} \int_{\cdot} g_\l(\A', \theta', \beta', {\Sigma}')) d\A' d\theta' d\beta' d{\Sigma}'.
\]
For any $\l \in (0, \infty)$ and any suitable set $B$,
\begin{align*}
&\inf_{
\substack{
\Sigma \in \R^{m \times m},
\\
V(\A, \theta, \beta) \le \l
}
} P\left( (\A, \theta, \beta, \Sigma), B \right)
\\
&= \inf_{V(\A, \theta, \beta) \le \l}
\int_B \pi_{n}(\A' | \theta', \beta', {\Sigma}')
\pi_{n}(\theta', \beta' | \A, {\Sigma}')
\pi_{n}({\Sigma}' | \A, \theta, \beta) d\A' d\theta' d\beta' d{\Sigma}'
\\
&\ge
\int_{B} g_\l(\A', \theta', \beta', {\Sigma}') 
d\A' d\theta' d\beta' d{\Sigma}'
\\
&= Z_{g_\l} \nu_\l(B).
\end{align*}

It remains to show a drift condition.
Fix $\A_0, \theta_0, \beta_0$, and fix $i \in 1, \ldots, n$.
Since $D_{i}$ is SPD, let $D_{i} = D_{i}^{1/2} D_{i}^{1/2}$, $D_{i}^{-1} = D_{i}^{-1/2} D_{i}^{-1/2}$ where $D_{i}^{1/2}, D_{i}^{-1/2}$ are SPD.
Using the identity
\[
d_{n, i}(\theta, \beta, \Sigma) 
= d_i + 
\left( \B \Sigma^{-1} \B^T + D_{i}^{-1} \right)^{-1}
\B \Sigma^{-1} \left( R_i  - \Theta^T M_i - \B^T d_i \right)
\]
and taking the expectation with respect to $\A_i | \A_0, \theta_0, \beta_0, \theta, \beta, \Sigma$
\begin{subequations}
\label{eq:expectation_A_norm}
\begin{align}
&\E\left[ \frac{1}{2} \norm{D_{i}^{-1/2} (\A_i - d_i)}_2^2 \big| \A_0, \theta_0, \beta_0, \theta, \beta, \Sigma 
\right]
\\
&= \frac{1}{2} \norm{D_{i}^{-1/2} \left( \B \Sigma^{-1} \B^T + D_{i}^{-1} \right)^{-1} 
\B \Sigma^{-1} \left( R_i  -  \Theta^T M_i - \B^T d_i \right)
}_2^2 
\\
&\hspace{.4cm}+ 
\frac{1}{2} \tr\left[ D_{i}^{-1/2} \left( \B \Sigma^{-1} \B^T + D_{i}^{-1} \right)^{-1} D_{i}^{-1/2} \right].
\end{align}
\end{subequations}

Using singular value decomposition from \citep[Theorem 2.6.3]{Horn2012}, choose matrices $U_i \in \R^{p \times p}$, $V_i \in \R^{m \times m}$ with $U_i^T U_i = U_i U_i^T = I_p$ and $V_i^T V_i$ $= V_i V_i^T = I_{m}$ and a rectangular diagonal matrix $\Sigma_{i} \in \R^{p \times m}$ with diagonal nonnegative singular values $(\sigma_{i, k})_{k}$ so that $D_i^{1/2} \B \Sigma^{-1/2} = U_i \Sigma_{i} V_i^T$.
Then
\begin{align}
\left( \B \Sigma^{-1} \B^T + D_{i}^{-1} \right)^{-1}
&= D_{i}^{1/2} (D_{i}^{1/2} \B \Sigma^{-1/2}) (D_{i}^{1/2} \B \Sigma^{-1/2})^T + I_p )^{-1} D_{i}^{1/2} \nonumber
\\
&= D_{i}^{1/2} U_i (\Sigma_{i} \Sigma_{i}^T + I_p )^{-1} U_i^T D_{i}^{1/2}.
\label{eq:Dbeta_identity}
\end{align}
Using \eqref{eq:Dbeta_identity} and properties of the trace
\begin{align*} 
\frac{1}{2} \tr\left[ D_{i}^{-1/2} \left( \B \Sigma^{-1} \B^T + D_{i}^{-1} \right)^{-1} D_{i}^{-1/2} \right]
&= \frac{1}{2} \tr\left[ 
\left( \Sigma_{i} \Sigma_{i}^T + I_p \right)^{-1} 
U_i^T U_i 
\right]
\\
&\le \frac{p}{2}.
\end{align*}
For $x \in [0, \infty)$ and $a \in (0, \infty)$, we have the inequality

\begin{align}
\frac{x}{x^2 + a} \le \frac{1}{2\sqrt{a}}.
\label{eq:inequality_ub}
\end{align}

Using inequalities \eqref{eq:Dbeta_identity} and \eqref{eq:inequality_ub}, the matrix norm is sub-multiplicative, and $\norm{U_i}_2 = 1$, we have
\begin{align*}
&\norm{
D_{i}^{-1/2} \left( \B \Sigma^{-1} \B^T + D_{i}^{-1} \right)^{-1} \B \Sigma^{-1/2} \Sigma^{-1/2}
}_2^2
\\
&= \norm{ U_i (\Sigma_{i} \Sigma_{i}^T + I_p )^{-1} U_i^T D_{i}^{1/2} \B \Sigma^{-1/2} \Sigma^{-1/2}}_2^2
\\
&\le \norm{(\Sigma_{i} \Sigma_{i}^T + I_p )^{-1} \Sigma_{i} \Sigma^{-1/2}}_2^2
\\
&\le \left[ \frac{ \sigma_i }{ \sigma_i^2 + 1 } \right]^2
\norm{\Sigma^{-1/2}}_2^2
\\
&\le \frac{\norm{\Sigma^{-1}}_2}{4}.
\end{align*}

Define the matrix $\tilde{X} = (d_1, \ldots, d_n)^T$.
Applying these upper bounds to \eqref{eq:expectation_A_norm} and combining for each $i \in 1, \ldots, n$,
\begin{align*}
&\E\left[
\frac{1}{2} \sum_{i = 1}^n \norm{ D_{i}^{-1/2} (\A_i - d_i) }_2^2 
\big| \A_0, \theta_0, \beta_0, \theta, \beta, \Sigma 
\right]
\\
&\le \frac{\norm{\Sigma^{-1}}_2}{8} \norm{ R - M\Theta - \tilde{X} \B }_F^2
+ \frac{p n}{2}.
\end{align*}
By convexity, for every $x, y$, 
$
\norm{x - y}_2^2 \le 2\norm{x}_2^2 + 2\norm{y}_2^2.
$
Since $C_{0}$ is SPD, let $C_{0} = C_{0}^{1/2} C_{0}^{1/2}$, $C_{0}^{-1} = C_{0}^{-1/2} C_{0}^{-1/2}$ where $C_{0}^{1/2}, C_{0}^{-1/2}$ are SPD.
Using convexity, and the matrix norm is sub-multiplicative, we have
\begin{align*}
\frac{1}{2} \norm{ R - M\Theta - \tilde{X} \B }_F^2
&\le \norm{ R }_F^2
+ \norm{ \mat{ M & \tilde{X} } C_0^{1/2}}_2^2 \norm{ C_0^{-1/2} (\theta, \beta)^T }_2^2.
\end{align*}
Therefore,
\begin{subequations}
\label{eq:drift_bound_1}
\begin{align}
&\E\left[
\frac{1}{2} \sum_{i = 1}^n \norm{ D_{i}^{-1/2} (\A_i - d_i) }_2^2 
\big| \A_0, \theta_0, \beta_0, \theta, \beta, \Sigma 
\right]
\\
&\le \frac{\norm{\Sigma^{-1}}_2}{8} \norm{ R - M \Theta - \tilde{X} \B }_F^2
+ \frac{p n}{2}
\\
&\le \norm{\Sigma^{-1}}_2 \frac{\norm{ R }_F^2}{4}
+ \norm{\Sigma^{-1}}_2 \frac{\norm{\mat{ M & \tilde{X} } C_0^{1/2}}_2^2}{4} \norm{ C_0^{-1/2} (\theta, \beta)^T }_2^2
+ \frac{p n}{2}.
\end{align}
\end{subequations}

Now taking the expectation with respect to $\theta, \beta | \A_0, \theta_0, \beta_0, \Sigma$
\begin{align*}
&\E \left[ \frac{1}{2} \norm{C_{0}^{-1/2} (\theta, \beta)^T}_2^2 \big| \Sigma, \A_0, \theta_0, \beta_0 \right]
\\
&= \frac{1}{2} 
\norm{
C_{0}^{-1/2} 
c_n(\A_0, \Sigma)
}_2^2 
+ \frac{1}{2} \tr(C_{0}^{-1/2} C_n(\A_0, \Sigma) C_{0}^{-1/2}).
\end{align*}
Using singular value decomposition \citep[Theorem 2.6.3]{Horn2012}, choose matrices $U \in \R^{mn \times mn}$, $V \in \R^{m(p + q) \times m(p + q)}$ with $U^T U = U U^T = I_{mn}$ and  $V^T V = V V^T = I_{m(p + q)}$ and a rectangular diagonal matrix $\Sigma_{\A_0} \in \R^{mn \times m(p + q)}$ with diagonal nonnegative singular values $(\sigma_{\A_0, k})_{k}$ so that $\Sigma^{-1/2} \otimes \mat{ M & \A_0 } C_{0}^{1/2} = U \Sigma_{\A_0} V^T$.
We then have
\begin{align}
C_n(\A_0, \Sigma)
&= \left( \Sigma^{-1} \otimes \mat{ M & \A_0 }^T \mat{ M & \A_0 } + C_{0}^{-1} \right)^{-1} \nonumber
\\
&= C_{0}^{1/2} V \left( \Sigma_{\A_0}^T \Sigma_{\A_0} + I_{m(p + q)} \right)^{-1} V^T C_{0}^{1/2}.
\label{eq:C_n_identity}
\end{align}
Using \eqref{eq:C_n_identity} and properties of the trace
\begin{align*}
\frac{1}{2} \tr(C_{0}^{-1/2} C_n(\A_0, \Sigma) C_{0}^{-1/2})
&\le \frac{1}{2} \max_{k} \left[ \left( \sigma_{\A_0, k}^2 + 1 \right)^{-1} \right] \tr(V^T V)
\\
&\le \frac{m (p + q)}{2}.
\end{align*}
Using convexity,
\begin{align*}
&\frac{1}{2} \norm{
C_{0}^{-1/2} 
c_n(\A_0, \Sigma)
}_2^2
\\
&= \frac{1}{2} \norm{C_{0}^{-1/2} C_n(\A_0, \Sigma) 
\left[ 
\left[ \Sigma^{-1} \otimes
\mat{ M & \A_0 }^T \right] \vectorize(R)
+ C_0^{-1} c_0
\right]
}_2^2
\\
&\le \norm{
C_{0}^{-1/2} C_n(\A_0, \Sigma) 
\left[ \Sigma^{-1/2} \otimes \mat{ M & \A_0 }^T  \right]
\left[ \Sigma^{-1/2} \otimes I_{mn} \right] 
}_2^2 
\norm{R}_F^2
\\
&\hspace{.5cm}+ \norm{
C_{0}^{-1/2} C_n(\A_0, \Sigma) C_0^{-1} c_0
}_2^2.
\end{align*}
Using the inequality \eqref{eq:inequality_ub} and the identity \eqref{eq:C_n_identity},
\begin{align*}
&\norm{
C_{0}^{-1/2} C_n(\A_0, \Sigma) 
\left[ \Sigma^{-1/2} \otimes \mat{ M & \A_0 } \right]^T
\left[ 
\Sigma^{-1/2} \otimes I_{mn} 
\right]
}_2^2
\\
&= \norm{
V \left( \Sigma_{\A_0}^T \Sigma_{\A_0} + I_{m(p + q)} \right)^{-1} V^T 
\left[ 
\Sigma^{-1/2} \otimes \mat{ M & \A_0 } C_0^{1/2}
\right]^T
\left[ 
\Sigma^{-1/2} \otimes I_{mn} 
\right]
}_2^2
\\
&= 
\norm{
V \left( \Sigma_{\A_0}^T \Sigma_{\A_0} + I_{m(p + q)} \right)^{-1} \Sigma_{\A_0}^T U^T
}_2^2
\norm{\Sigma^{-1}}_2
\\
&\le \norm{V}_2^2 \norm{
\left( \Sigma_{\A_0}^T \Sigma_{\A_0} + I_{m(p + q)} \right)^{-1} 
\Sigma_{\A_0}^T}_2^2 \norm{U^T}_2^2
\norm{\Sigma^{-1}}_2
\\
&\le \max_{k} \left( \frac{\sigma_{\A_0, k}}{\sigma_{\A_0, k}^2 + 1} \right)^2
\norm{\Sigma^{-1}}_2
\\
&\le \frac{\norm{\Sigma^{-1}}_2}{4}.
\end{align*}
Using \eqref{eq:C_n_identity},
\begin{align*}
\norm{
C_{0}^{-1/2} C_n(\A_0, \Sigma) C_0^{-1} c_0
}_2^2
&= \norm{
V \left( \Sigma_{\A_0}^T \Sigma_{\A_0} + I_{m(p + q)} \right)^{-1} V^T C_0^{-1/2} c_0
}_2^2
\\
&\le \norm{V}_2^2 \norm{
\left( \Sigma_{\A_0}^T \Sigma_{\A_0} 
+ I_{m(p + q)} \right)^{-1} }_2^2 \norm{V^T}_2^2
c_0^T C_0^{-1} c_0
\\
&\le c_0^T C_0^{-1} c_0.
\end{align*}
Combining the upper bounds
\begin{align}
\E \left[ \frac{1}{2} \norm{ C_{0}^{-1/2} (\theta, \beta)^T }_2^2 \big| \Sigma, \A_0, \theta_0, \beta_0 \right]
&\le
\frac{\norm{R}_F^2}{4} \norm{\Sigma^{-1}}_2
+ c_0^T C_0^{-1} c_0
+ \frac{m(p + q)}{2}.
\label{eq:drift_bound_2}
\end{align}

Now using \eqref{eq:drift_bound_1} and \eqref{eq:drift_bound_2} and taking the iterated expectation with respect to $\theta, \beta | \A_0, \theta_0, \beta_0, \Sigma$,
\begin{subequations}
\label{eq:drift_bound_3}
\begin{align}
&\E\left[ \frac{1}{2} \sum_{i = 1}^n \norm{ D_{i}^{-1/2} (\A_i - d_i) }_2^2
\big| \Sigma, \A_0, \theta_0, \beta_0 \right]
\\
&\le \norm{\Sigma^{-1}}_2
\\
&\hspace{.4cm}\times 
\left[ 
\frac{\norm{ R }_F^2}{4}
+ \frac{\norm{ \mat{ M & \tilde{X} } C_0^{1/2}}_2^2 m (p + q)}{4}
+ \frac{\norm{\mat{ M & \tilde{X} } C_0^{1/2}}_2^2}{2} c_0^T C_0^{-1} c_0
\right] 
\\
&\hspace{.4cm}+ \norm{\Sigma^{-1}}_2^2
\frac{\norm{ \mat{ M & \tilde{X} } C_0^{1/2}}_2^2 \norm{R}_F^2}{8}
+ \frac{p n}{2}.
\end{align}
\end{subequations}
Since $\Sigma^{-1}$ has a Wishart distribution and using properties of the trace,
\begin{align*}
\E \norm{\Sigma^{-1}}_2
&\le \tr\left[ 
\E\left( \Sigma^{-1} \right) \right]
\\
&= (n + a_0) \tr\left[
\left( ( R - M \Theta_0 - \A_0 \B_0 )^T ( R - M \Theta_0 - \A_0 \B_0 )
+ B_0 \right)^{-1}
\right]
\\
&\le (n + a_0) tr[B_0^{-1}].
\end{align*}
Similarly, we use the second moment formula of the Wishart \citep{Letac:Massam:2004} to get the upper bound,
\begin{align*}
\E \norm{\Sigma^{-1}}_2^2
&\le \tr\left[ \E\left( \Sigma^{-2} \right) \right]
\\
&= (n + a_0) 
\tr\left[
\left( ( R - M \Theta_0 - \A_0 \B_0 ) ( R - M \Theta_0 - \A_0 \B_0 )^T
+ B_0 \right)^{-1}
\right]^2
\\
&\hspace{.5cm}+ (n + a_0) (n + a_0 + 1) 
\\
&\hspace{.5cm}\times
\tr\left[
\left( ( R - M \Theta_0 - \A_0 \B_0 ) ( R - M \Theta_0 - \A_0 \B_0 )^T
+ B_0 \right)^{-2}
\right]
\\
&\le (n + a_0) (n + a_0 + 1)  tr[B_0^{-1}]^2.
\end{align*}

Taking the iterated expectation with respect to $\Sigma | \A_0, \theta_0, \beta_0$ in \eqref{eq:drift_bound_3}, there is a constant $L \in (0, \infty)$ so that the drift condition is satisfied with
\[
\E\left[ V(\A, \theta, \beta) 
| \A_0, \theta_0, \beta_0 \right]
\le L.
\]
\end{proof}

\subsection{Bayesian EIV regression with errors in the features}

Using Theorem~\ref{theorem:ge}, we develop geometrically ergodic Gibbs samplers for Bayesian EIV regression with additive Gaussian error in the features.
For the remainder, we write the observed data as $Y = (y_1, \ldots, y_n)^T \in \R^{n \times m}$, $X = (x_1, \ldots, x_n)^T \in \R^{n \times p}$, and $Z = (Z_1, \ldots, Z_n)^T \in \R^{n \times q}$.
Consider the Bayesian EIV regression \eqref{eq:EIV_regression_model_X} with Berkson errors and priors \eqref{eq:EIV_priors} and \eqref{eq:EIV_priors_A}.
We will write the posterior density $\pi_{n}$ for this Bayesian model as
\begin{align*}
&\pi_{n}(\A, \theta, \beta, \Sigma)
\\
&\propto
\left( \frac{1}{\det(\Sigma)} \right)^{(n + a_0 + m + 1)/2}
\exp\left[
- \frac{1}{2} \tr[ \Sigma^{-1} B_0 ]
\right]
\\
&\hspace{.5cm}
\times
\exp\left[
- \frac{1}{2}
\sum_{i = 1}^n ( y_i - \Theta^T Z_i - \B^T \A_{i} )^T \Sigma^{-1} ( y_i - \Theta^T Z_i - \B^T \A_{i} )
\right]
\\
&\hspace{.5cm}
\times 
\exp\left( -\frac{1}{2} \sum_{i = 1}^n (\A_i - x_i)^T V_i^{-1} (\A_i - x_i) \right)
\\
&\hspace{.5cm}
\times 
\exp\left( -\frac{1}{2} ((\theta, \beta)^T - j_0)^T J_0^{-1} ((\theta, \beta)^T - j_0) \right).
\end{align*}
This posterior density is a special case of the general density \eqref{eq:general_posterior} choosing $M \equiv Z$, $R \equiv Y$, $c_0, C_0 \equiv j_0, J_0$, and $d_i, D_i \equiv x_i, V_i$.

We can define a 3-variable deterministic scan Gibbs sampler which generates a Markov chain $(\A_t, \theta_t, \beta_t, \Sigma_t)_{t=0}^{\infty}$ for this posterior density as a special case of the Gibbs sampler constructed in Section~\ref{section:gibbs_sampler}.
Initialize $(\A_0, \theta_0, \beta_0, \Sigma_0)$ and for $t \in 1, \ldots$,
\begin{enumerate} 
\item Generate 
$
\Sigma_t | \A_{t-1}, \theta_{t-1}, \beta_{t-1}
\sim \W^{-1}\left(
n + a_0, 
B_{n, t}
\right)
$
where 
\[
B_{n, t}
= ( Y - Z \Theta_{t-1} - \A_{t-1} \B_{t-1} )^T (Y - Z \Theta_{t-1} - \A_{t-1} \B_{t-1} )
+ B_0
\]
\item Generate 
$
(\theta_t, \beta_t)^T | \A_{t-1}, \Sigma_{t}
\sim N_{m(p + q)}(
j_{n, t},
J_{n, t}
)
$
where
\begin{align*}
&J_{n, t}
= \left( 
\Sigma^{-1}_{t} \otimes \mat{ Z & \A_{t-1} }^T \mat{ Z & \A_{t-1} } + J_{0}^{-1} 
\right)^{-1}
\\
&j_{n, t}
= J_{n, t} \left[ 
\left[ \Sigma^{-1}_{t} \otimes \mat{ Z & \A_{t-1} }^T \right] \vectorize(Y) 
+ J_0^{-1} j_0
\right]
\end{align*}

\item Generate 
$
\A_{i, t} | \theta_{t}, \beta_{t}, \Sigma_{t}
\sim N_p(
k_{n, i, t},
K_{n, i, t}
)
, i \in 1, \ldots, n
$
where
\begin{align*}
&K_{n, i, t}
= \left( \B_{t} \Sigma^{-1}_t \B_{t}^T + V_{i}^{-1} \right)^{-1}
\\
&k_{n, i, t}
= K_{n, i, t} 
\left[ V_i^{-1} x_i + \B_{t} \Sigma^{-1}_{t} \left( y_i  - \Theta_{t}^T Z_i \right) \right].
\end{align*}
\end{enumerate}
Applying Theorem~\ref{theorem:ge}, we have the following result.
\begin{corollary}
The 3-variable Gibbs sampler $(\A_t, \theta_t, \beta_t, \Sigma_t)_{t=0}^{\infty}$ for the posterior in Bayesian EIV regression \eqref{eq:EIV_regression_model_X} with Berkson errors and priors \eqref{eq:EIV_priors} and \eqref{eq:EIV_priors_A} is geometrically ergodic.
\end{corollary}

Now consider Bayesian EIV regression \eqref{eq:EIV_regression_model_X} with additive Gaussian error in $X_i$ using classical errors and priors \eqref{eq:EIV_priors} and \eqref{eq:EIV_priors_A}.
The posterior has density
\begin{align*}
&\pi_{n}(\A, \theta, \beta, \Sigma)
\\
&\propto
\left( \frac{1}{\det(\Sigma)} \right)^{(n + a_0 + m + 1)/2}
\exp\left[
- \frac{1}{2} \tr[\Sigma^{-1} B_0]
\right]
\\
&\hspace{.5cm}
\times
\exp\left[
- \frac{1}{2} \sum_{i = 1}^n ( y_i - \Theta^T Z_i - \B^T \A_i )^T \Sigma^{-1}
( y_i - \Theta^T Z_i - \B^T \A_i )
\right]
\\
&\hspace{.5cm}
\times
\exp\left( -\frac{1}{2} \sum_{i = 1}^n  (\A_i - k_i')^T (V_{i}^{-1} + K_i^{-1}) (\A_i - k_i') \right)
\\
&\hspace{.5cm}
\times 
\exp\left( -\frac{1}{2} ((\theta, \beta)^T - j_0)^T J_0^{-1} ((\theta, \beta)^T - j_0) \right)
\end{align*}
where
$
k_i' = (V_{i}^{-1} + K_i^{-1})^{-1} 
\left[ V_{i}^{-1} x_i + K_i^{-1} k_i
\right].
$
The posterior density is also a special case of the general density \eqref{eq:general_posterior} when $Z \equiv M$, $R \equiv Y$, and $c_0, C_0 \equiv j_0, J_0$, and $d_i, D_i \equiv k_i', (V_{i}^{-1} + K_i^{-1})^{-1}$.

We define a 3-variable deterministic scan Gibbs sampler similarly. 
Initialize $(\A_0, \theta_0, \beta_0, \Sigma_0)$ and for $t \in 1, \ldots$,
\begin{enumerate} 
\item Generate 
$
\Sigma_t | \A_{t-1}, \theta_{t-1}, \beta_{t-1}
\sim \W^{-1}\left( 
n + a_0, 
B_{n, t}
\right)
$

\item Generate 
$
(\theta_t, \beta_t)^T | \A_{t-1}, \Sigma_{t}
\sim N_{m(p + q)}(
j_{n, t},
J_{n, t}
)
$

\item Generate 
$
\A_{i, t} | \theta_{t}, \beta_{t}, \Sigma_{t}
\sim N_p(
k'_{n, i, t},
K'_{n, i, t}
)
, i \in 1, \ldots, n
$
where
\begin{align*}
&K'_{n, i, t}
= \left( \B_{t} \Sigma^{-1}_{t} \B_{t}^T + V_{i}^{-1} + K_i^{-1} \right)^{-1}
\\
&k'_{n, i, t}
= K'_{n, i, t} 
\left[ V_i^{-1} x_i + K_i^{-1} k_i + \B_{t} \Sigma^{-1}_{t} \left( y_i  - \Theta_{t}^T Z_i \right) \right].
\end{align*}
\end{enumerate}
We also have the following as a direct result of Theorem~\ref{theorem:ge}.
\begin{corollary}
The 3-variable Gibbs sampler $(\A_t, \theta_t, \beta_t, \Sigma_t)_{t=0}^{\infty}$ for the posterior in Bayesian EIV regression \eqref{eq:EIV_regression_model_X} with classical errors and priors \eqref{eq:EIV_priors} and \eqref{eq:EIV_priors_A} is geometrically ergodic.
\end{corollary}

\subsection{Bayesian EIV regression with errors in the response and features}

Similar to the previous section, we develop geometrically ergodic Gibbs samplers for Bayesian EIV regression with additional additive Gaussian error in the features and response.
Consider the Bayesian EIV regression \eqref{eq:EIV_regression_model_X_Y} with Berkson errors in $X_i$ and additional error in $Y_i$ along with priors \eqref{eq:EIV_priors} and \eqref{eq:EIV_priors_A}.
Let $\mathcal{V} = ( \V_1, \ldots, \V_n)^T \in \R^{n \times m}$ and $\nu = \vectorize(\V)$, and let $U_0 = \text{blockdiag}(U_i) \in \R^{mn \times mn}$.
The Bayesian posterior $\Pi_n$ has density
\begin{align*}
&\pi_{n}(\A, \nu, \theta, \beta, \Sigma)
\\
&\propto \left( \frac{1}{\det(\Sigma)} \right)^{(n + a_0 + m + 1)/2}
\exp\left[
- \frac{1}{2} \tr[\Sigma^{-1} B_0 ]
\right]
\\
&\hspace{.4cm}
\times
\exp\left[
-\frac{1}{2} \sum_{i = 1}^n ( \V_i - \Theta^T Z_i - \B^T \A_i)^T \Sigma^{-1} ( \V_i - \Theta^T Z_i - \B^T \A_i)
\right]
\\
&\hspace{.4cm}
\times \exp\left( -\frac{1}{2} \sum_{i = 1}^n (\V_i - y_i)^T U_i^{-1} (\V_i - y_i) \right)
\\
&\hspace{.4cm}
\times 
\exp\left( -\frac{1}{2} \sum_{i = 1}^n (\A_i - x_i)^T V_i^{-1} (\A_i - x_i) \right)
\\
&\hspace{.4cm}
\times \exp\left( -\frac{1}{2} ((\theta, \beta)^T - j_0)^T J_0^{-1} ((\theta, \beta)^T - j_0) \right).
\end{align*}
This posterior density is a special case of the density \eqref{eq:general_posterior} when redefining $\tilde{\theta} \equiv (\nu, \theta)^T$, $M \equiv \mat{-I & Z}$, $r \equiv 0$, $c_0 = (Y, j_0)^T$,
\[
C_0 \equiv  
\mat{
U_0 & 0 \\
0 & J_0
},
\]
and $d_i, D_i \equiv x_i, V_i$.

We define a 3-variable deterministic scan Gibbs sampler which generates a Markov chain $(\A_t, \nu_t, \theta_t, \beta_t, \Sigma_t)_{t=0}^{\infty}$ for this posterior density.
Initialize $\A_0$, $\nu_0$, $\theta_0$, $\beta_0$, $\Sigma_0$ and for $t \in 1, \ldots$,
\begin{enumerate} 
\item Generate 
$
\Sigma_t | \A_{t-1}, \nu_{t-1}, \theta_{t-1}, \beta_{t-1}
\sim \W^{-1}\left(
n + a_0, 
B'_{n, t}
\right)
$
where 
\[
B_{n, t}'
= ( \V_{t-1} - Z \Theta_{t-1} - \A_{t-1} \B_{t-1} )^T ( \V_{t-1} - Z \Theta_{t-1} - \A_{t-1} \B_{t-1} ) 
+ B_0
\]
\item Generate 
$
(\nu_t, \theta_t, \beta_t)^T | \A_{t-1}, \Sigma_{t}
\sim N_{p + q}(
j'_{n, t},
J'_{n, t}
)
$
where
\begin{align*}
&J'_{n, t}
= \left( 
\Sigma^{-1}_{t} \otimes \mat{ -I & Z & \A_{t-1}}^T \mat{-I & Z & \A_{t-1} }
+ 
\mat{
U_0^{-1} & 0 \\
0 & J_0^{-1}
}
\right)^{-1}
\\
&j'_{n, t}
= J_{n, t}' 
\mat{
U_0^{-1} & 0 \\
0 & J_0^{-1}
}
(\vectorize(Y), j_0)^T
\end{align*}

\item Generate 
$
\A_{i, t} | \nu_t, \theta_{t}, \beta_{t}, \Sigma_{t}
\sim N_p(
k''_{n, i, t},
K''_{n, i, t}
)
, i \in 1, \ldots, n
$
where
\begin{align*}
&K''_{n, i, t}
= \left( \B_{t} \Sigma^{-1}_{t} \B_{t}^T + V_{i}^{-1} \right)^{-1}
\\
&k''_{n, i, t}
= K''_{n, i, t}
\left[ V_i^{-1} x_i +  \B_{t} \Sigma^{-1}_{t} \left( \V_{i, t}  - \Theta_{t}^T Z_i \right) \right].
\end{align*}
\end{enumerate}
Using Theorem~\ref{theorem:ge}, we have the following result.
\begin{corollary}
The 3-variable Gibbs sampler $(\A_t, \nu_t, \theta_t, \beta_t, \Sigma_t)_{t=0}^{\infty}$ for Bayesian EIV regression \eqref{eq:EIV_regression_model_X_Y} with Berkson errors and priors \eqref{eq:EIV_priors} and \eqref{eq:EIV_priors_A} is geometrically ergodic.
\end{corollary}

Now consider the Bayesian EIV regression \eqref{eq:EIV_regression_model_X_Y} with classical errors in $X_i$ and additional error in $Y_i$ with priors \eqref{eq:EIV_priors} and \eqref{eq:EIV_priors_A}.
The posterior $\Pi_n$ for this Bayesian model has density
\begin{align*}
&\pi_{n}(\A, \nu, \theta, \beta, \Sigma)
\\
&\propto \left( \frac{1}{\det(\Sigma)} \right)^{(n + a_0 + m + 1)/2}
\exp\left[
- \frac{1}{2} \tr[ \Sigma^{-1} B_0 ]
\right]
\\
&\hspace{.4cm}
\times 
\exp\left[
-\frac{1}{2} \sum_{i = 1}^n ( \V_i - \Theta^T Z_i - \B^T \A_i )^T \Sigma^{-1} ( \V_i - \Theta^T Z_i - \B^T \A_i )
\right]
\\
&\hspace{.4cm}
\times 
\exp\left( -\frac{1}{2} \sum_{i = 1}^n (\V_i - y_i)^T U_i^{-1} (\V_i - y_i) \right)
\\
&\hspace{.4cm}
\times 
\exp\left( -\frac{1}{2} \sum_{i = 1}^n (\A_i - k_i')^T (V_{i}^{-1} + K_i^{-1}) (\A_i - k_i') \right)
\\
&\hspace{.4cm}
\times \exp\left( -\frac{1}{2} ((\theta, \beta)^T - j_0)^T J_0^{-1} ((\theta, \beta)^T - j_0) \right).
\end{align*}
This posterior density is also a special case of the density \eqref{eq:general_posterior} when redefining $\tilde{\theta} \equiv (\nu, \theta)^T$, $M \equiv \mat{ -I & Z }$, $R \equiv 0$, $c_0 = (y, j_0)^T$, 
\[
C_0 \equiv  
\mat{
U_0 & 0 \\
0 & J_0
},
\]
and $d_i, D_i \equiv k_i', (V_{i}^{-1} + K_i^{-1})^{-1}$.

We define a 3-variable deterministic scan Gibbs sampler similarly. 
Initialize $(\A_0, \nu_0, \theta_0, \beta_0, \Sigma_0)$ and for $t \in 1, \ldots$,
\begin{enumerate} 
\item Generate 
$
\Sigma_t | \A_{t-1}, \nu_{t-1}, \theta_{t-1}, \beta_{t-1}
\sim \W^{-1}\left(
n + a_0, 
B'_{n, t}
\right)
$

\item Generate 
$
(\nu_t, \theta_t, \beta_t)^T | \A_{t-1}, \Sigma_{t}
\sim N_{p + q}(
j'_{n, t},
J'_{n, t}
)
$

\item Generate 
$
\A_{i, t} | \nu_t, \theta_{t}, \beta_{t}, \Sigma_{t}
\sim N_p(
k'''_{n, i, t},
K'''_{n, i, t}
)
, i \in 1, \ldots, n
$
where
\begin{align*}
&K'''_{n, i, t}
= \left( \B_{t} \Sigma_{t} \B_{t}^T + V_{i}^{-1} + K_i^{-1} \right)^{-1}
\\
&k'''_{n, i, t}
= K'''_{n, i, t}
\left[ V_i^{-1} x_i + K_i^{-1} k_i +  \B_{t} \Sigma^{-1}_{t} \left( \V_{i, t}  - \Theta_{t}^T Z_i \right) \right].
\end{align*}
\end{enumerate}
Using Theorem~\ref{theorem:ge}, we have the following result.
\begin{corollary}
The 3-variable Gibbs sampler $(\A_t, \nu_t, \theta_t, \beta_t, \Sigma_t)_{t=0}^{\infty}$ for Bayesian EIV regression \eqref{eq:EIV_regression_model_X_Y} with classical errors and priors \eqref{eq:EIV_priors} and \eqref{eq:EIV_priors_A} is geometrically ergodic.
\end{corollary}

\section{Simulations}
\label{section:simulation}

\subsection{Limitations of the Gibbs sampler in large problem sizes}
\label{section:toy_example}

Theoretically, we developed a \textit{qualitative} convergence result for the Gibbs sampler in Bayesian multivariate EIV regression.
It is important in practice to understand the relationship between scaling of the problem size and the estimation reliability from the Gibbs sampler.
We look at artificially generated data to empirically demonstrate the dependence of the Gibbs sampler when the the dimension of the response $m$ and the dimension of the features $p$ are increasing in configurations $(m, p) = (1, 1), (2, 7), (3, 7)$ in the Bayesian posterior.
Artificial data is generated according to the multivariate EIV Berkson linear regression model for $i =1, \ldots, 50$ with
\begin{align*}
&\Sigma \sim \W^{-1}(m, 10^{-3} I)
\\
&(\theta, \beta)^T \sim N_{m +  m p}(0, 10^{3} I)
\\
&X_i | \A_i \sim N_p(\A_i, .2 I)
\\
&Y_i | \A_i, \Theta, \B, \Sigma \sim N_{m}(\Theta^T 1 + \B^T \A_i, \Sigma)
\end{align*}

We simulate $T = 10^5$ MCMC realizations from the Gibbs sampler in each configuration using $10^4$ realizations for burn-in and analyze diagnostics for $\beta_t$ taking values in $\R^{m p}$.
We independently replicate the simulation $5$ times to reduce variability.
Geometric ergodicity guarantees the properly scaled and summed samples from the Gibbs sampler 
\[
\frac{1}{\sqrt{T}} \sum_{t=1}^T \left[ \beta_t - \int \beta \Pi_n(d\beta) \right]
\to N(0, \Sigma_*)
\]
as $T \to \infty$ in distribution where $\Sigma_*$ is a SPD covariance matrix.
The central limit theorem can be seen to hold in this case due to the established drift condition with the drift function.
Figure~\ref{figure:toy_example:a} and Figure~\ref{figure:toy_example:b} plot the largest and smallest eigenvalues of a batch means estimate to the multivariate standard error matrix ${\Sigma_{*} }^{1/2}$ in the central limit theorem. 
This simulation shows an increase in the largest eigenvalue in larger size problems which can lead to suggesting more iterations are needed for appropriate estimation in practice.
Figure~\ref{figure:toy_example:c} plots the multivariate effective sample size \citep{dootika:et:al:2019}.
We can also see a relatively sharp decrease in the estimation of the effective sample size as the problem size increases suggesting the algorithm should be run for many iterations even in moderately sized problems.
This simulation demonstrates that even though the algorithm is \textit{always} geometrically ergodic and can scale reasonably well to larger problem sizes, the Gibbs sampler generally requires many more iterations for reliable estimation even in moderately sized problems. 
Simulation code is available using Python \citep{python1995} and \citep{numpy2020} for matrix calculations at \url{https://github.com/austindavidbrown/BayesEIV}.

\begin{figure}[t]
\centering
\begin{subfigure}{.32\textwidth}
  \centering
  \includegraphics[width=\linewidth]{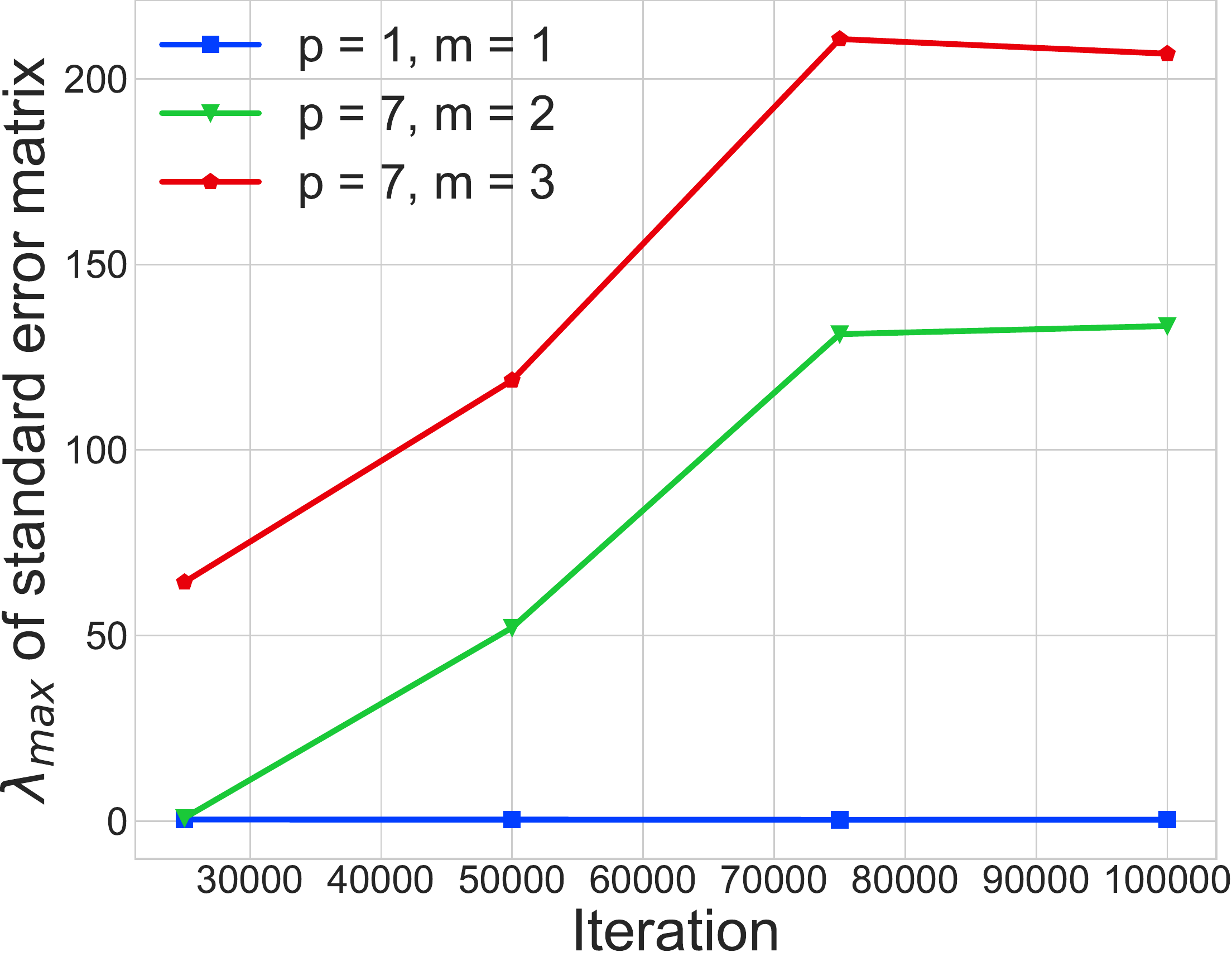}
  \caption{}
  \label{figure:toy_example:a}
\end{subfigure}
\begin{subfigure}{.32\textwidth}
  \centering
  \includegraphics[width=\linewidth]{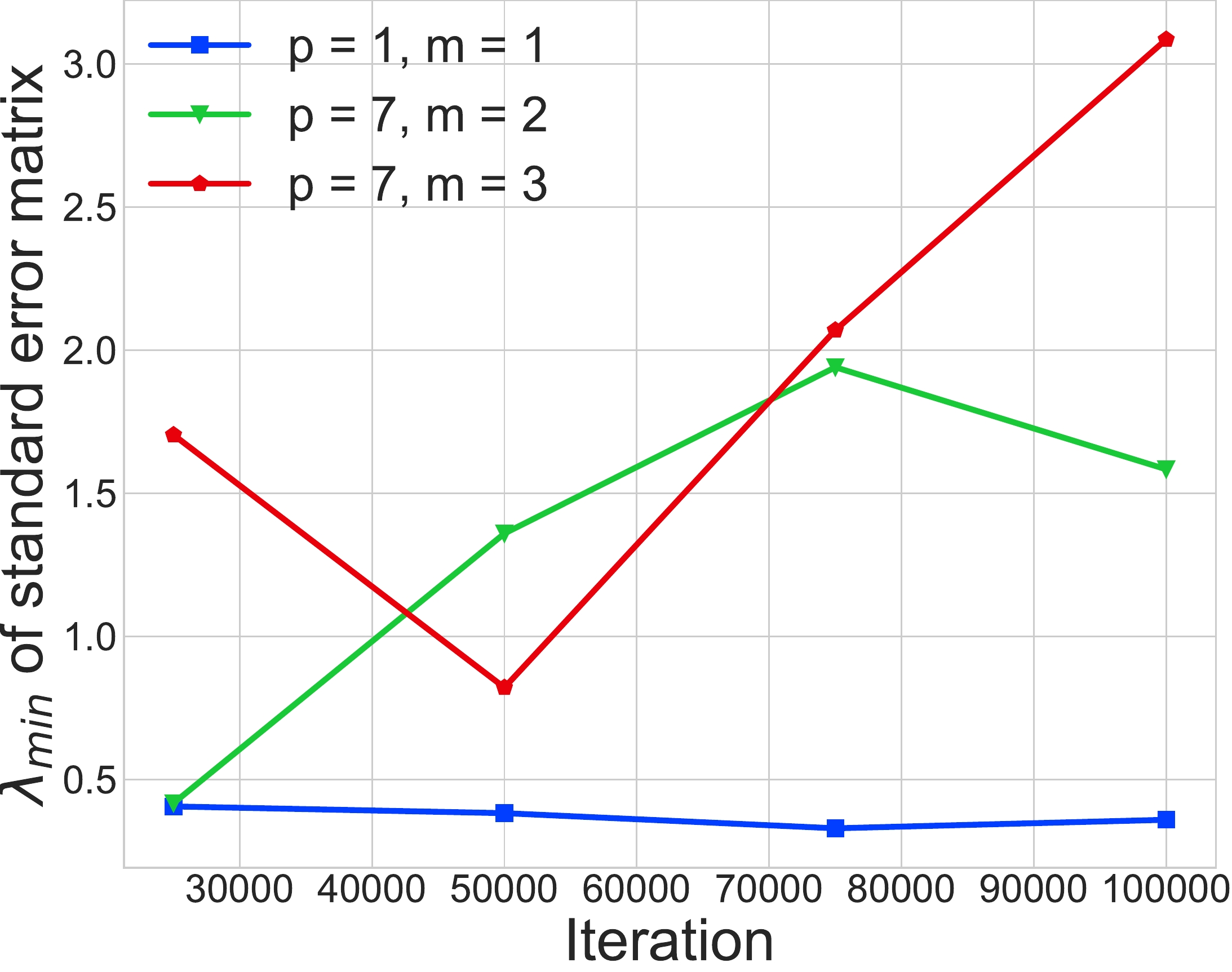}
  \caption{}
  \label{figure:toy_example:b}
\end{subfigure}
\begin{subfigure}{.32\textwidth}
  \centering
  \includegraphics[width=\linewidth]{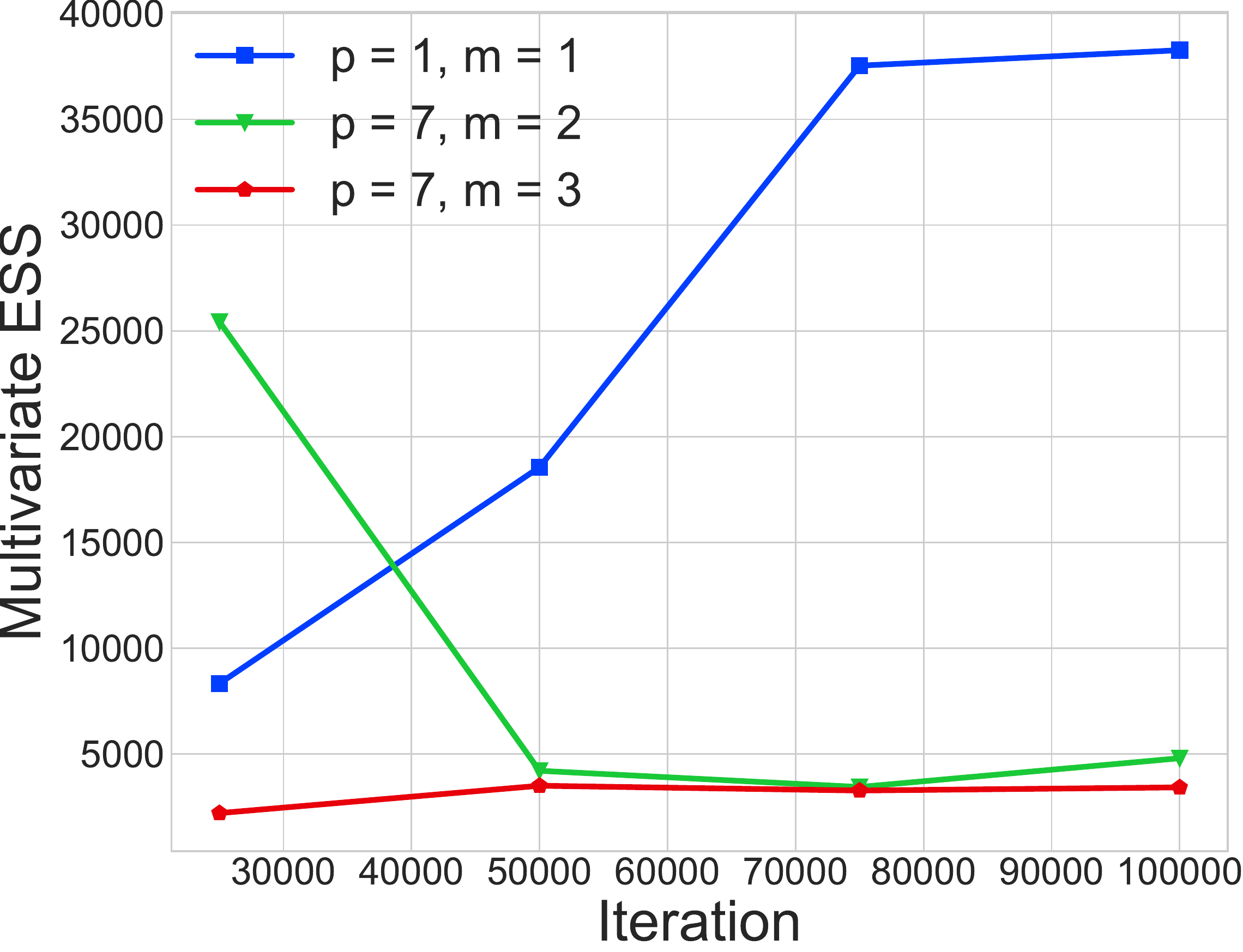}
  \caption{}
  \label{figure:toy_example:c}
\end{subfigure}
\caption{(a), (b) Largest and smallest eigenvalues of the MCMC standard error matrix targeting the average of $\beta$ for iterations of the Gibbs sampler and (c) the multivariate effective sample size for iterations of the Gibbs sampler}\label{figure:toy_example}
\end{figure}

\subsection{Robustness to model misspecification}

Although the multivariate Bayesian model for EIV accounts for additional error in the features, this error can be misspecified.
In particular, the error $X_i | \A_i$ from the model in Section~\ref{section:toy_example} may be a multivariate $t$ distribution with heavier tails in practical problems.
We are interested to empirically study the robustness of the convergence of the Gibbs sampler to misspecification in this modeling error.
Denote $t_d(v, m, V)$ as a multivariate $t$ distribution in dimension $d$ with $v$ degrees of freedom, location vector $m$, and scale matrix $V$.
With $df$ denoting the degrees of freedom, artificial data is generated according to the misspecified multivariate EIV Berkson linear regression model for $i =1, \ldots, 50$ with
\begin{align*}
&\Sigma \sim \W^{-1}(3, 10^{-3} I)
\\
&(\theta, \beta)^T \sim N_{3 +  9}(0, 10^{3} I)
\\
&X_i | \A_i \sim \text{t}_3(\text{df}, \A_i, .2 I)
\\
&Y_i | \A_i, \Theta, \B, \Sigma \sim N_{3}(\Theta^T 1 + \B^T \A_i, \Sigma).
\end{align*}

We look at compare more dispersed tail behavior with $df = 2$ and less dispersed tail behavior with $df = 10$.
We replicate a simulation $5$ times where we simulate $T = 10^5$ MCMC realizations from the Gibbs sampler in each configuration using $10^4$ realizations for burn-in and analyze diagnostics for $\beta_t$.
Similar to the previous simulation in Section~\ref{section:toy_example}, Figure~\ref{figure:robust_example:a} and Figure~\ref{figure:robust_example:b} plot the largest and smallest eigenvalues of a batch means estimate to the multivariate standard error matrix and
Figure~\ref{figure:robust_example:c} plots the multivariate effective sample size \citep{dootika:et:al:2019}.
We can see similar behavior in the estimation from the Gibbs sampler based on the both smaller and larger degrees of freedom. 
The simulation results suggest the Gibbs sampler is reasonably robust to misspecification of the tails in the error distribution of the features for $X_i | \A_i$. 

\begin{figure}[t]
\centering
\begin{subfigure}{.32\textwidth}
  \centering
  \includegraphics[width=\linewidth]{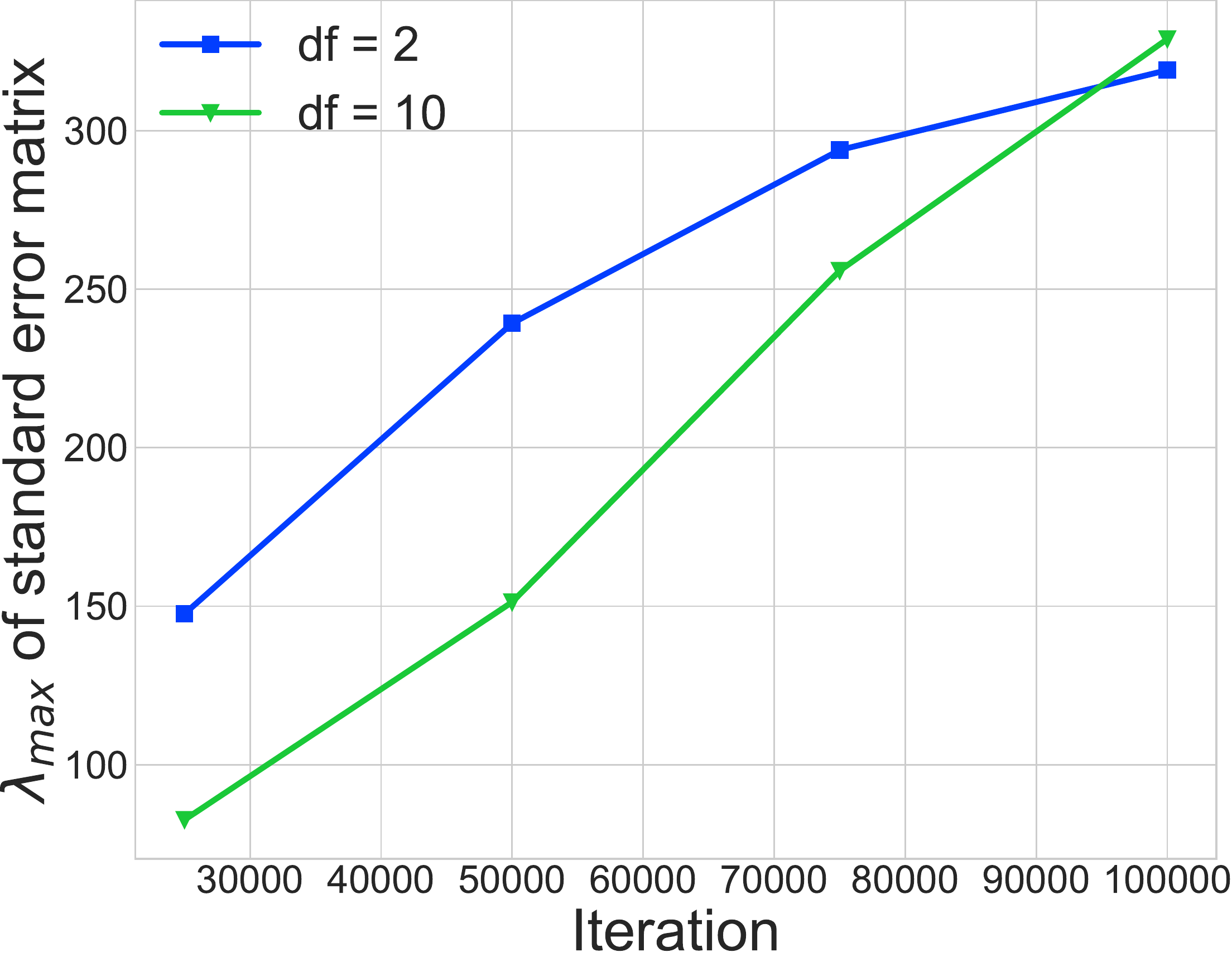}
  \caption{}
  \label{figure:robust_example:a}
\end{subfigure}
\begin{subfigure}{.32\textwidth}
  \centering
  \includegraphics[width=\linewidth]{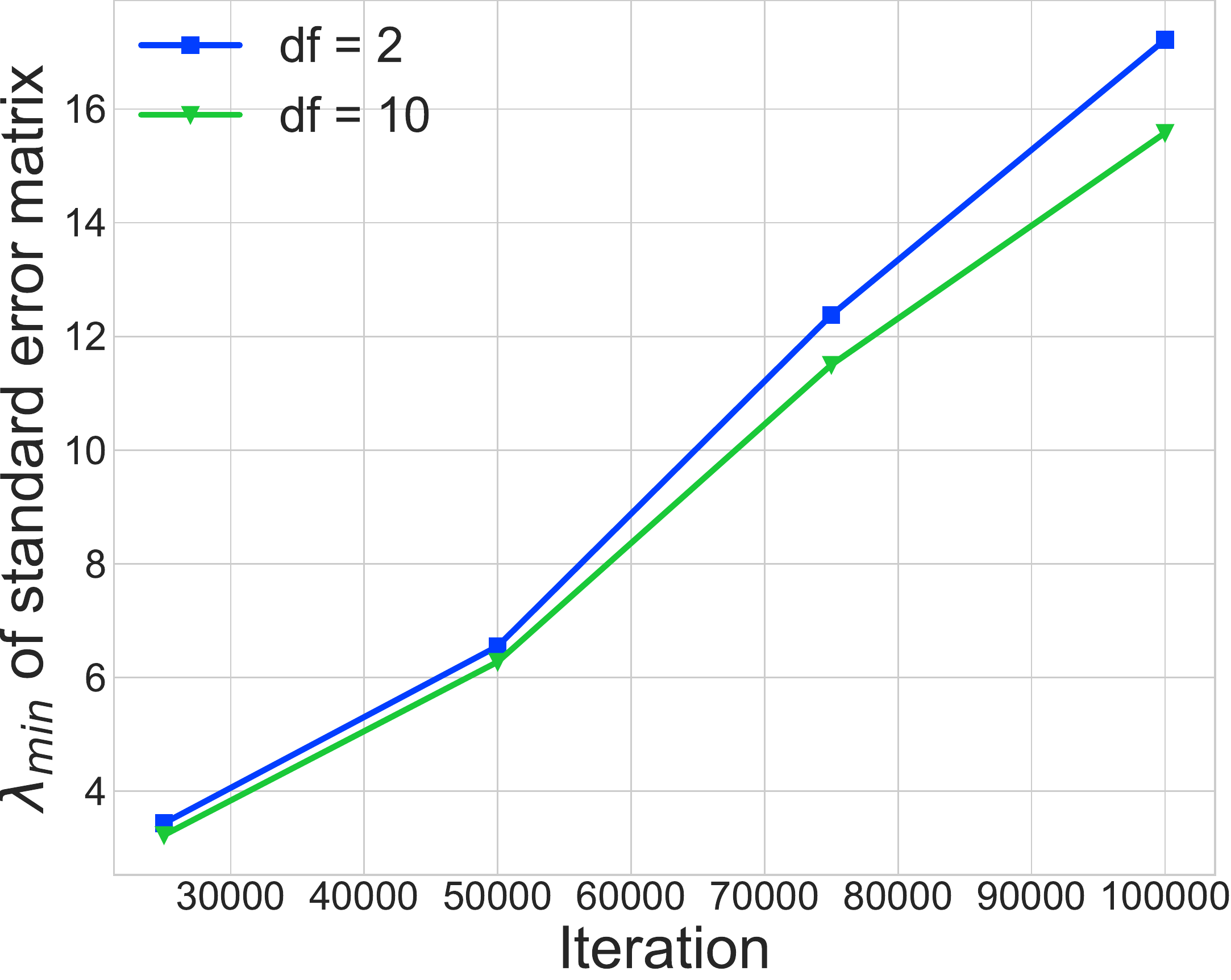}
  \caption{}
  \label{figure:robust_example:b}
\end{subfigure}
\begin{subfigure}{.32\textwidth}
  \centering
  \includegraphics[width=\linewidth]{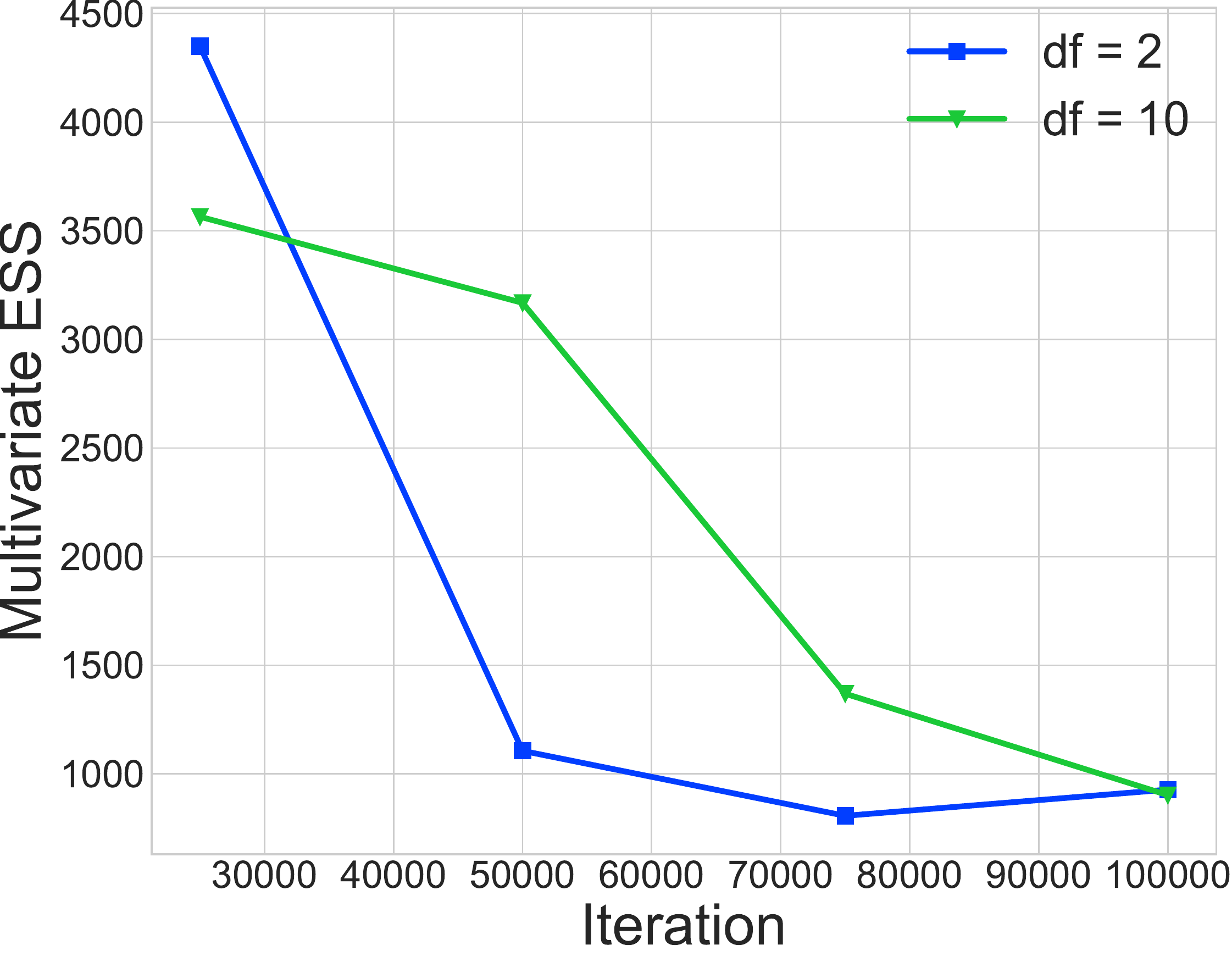}
  \caption{}
  \label{figure:robust_example:c}
\end{subfigure}
\caption{(a), (b) Largest and smallest eigenvalues of the MCMC standard error matrix (c) the multivariate effective sample size for iterations of the Gibbs sampler}\label{figure:robust_example}
\end{figure}

\section{Real-data example: measurement error in astrophysics}
\label{section:astro_example}

We look at Bayesian EIV linear regression proposed and analyzed in \citep{Harris2014, Hilbe2017}.
The dataset consists of the central galaxy supermassive black hole mass and the stellar bulge velocity dispersion from $n = 46$ different galaxies \citep{Harris2014}.
The response $Y_i$ is the logarithm of the observed central black hole mass and the predictor variable $X_i$ is the logarithm of the observed velocity dispersion.
The measurement errors are known beforehand and denoted by $\e_{Y_i}$ and $\e_{X_i}$ for both the response and predictor variables.
The EIV linear regression model studied in \citep{Hilbe2017} folllows
\begin{align*}
&\sigma^2 \sim \text{Inverse-gamma}(10^{-3}, 10^{-3})
\\
&\alpha \sim N_1(0, 10^{3}), \beta \sim N_1(0, 10^{3})
\\
&\A_i \sim N_1(0, 10^{3}), X_i | \A_i \sim N_1(\A_i, \e_{X_i}^2)
\\
&\V_i | \A_i, \theta, \beta, \sigma^2 \sim N_1(\theta + \A_i \beta, \sigma^2), Y_i | \V_i \sim N_1(\V_i, \e_{y_i}^2).
\end{align*}

We generate $10^5$ MCMC realizations from the Gibbs sampler.
Figure~\ref{figure:astro_example} plots the autocorrelation, estimates to the standard errors in the central limit theorem, and effective sample sizes from these realizations.
The autocorrelations are computed up to lag $20$.
Overall, we see the Gibbs sampler performs well. 
However, the standard error and effective sample size plots suggest that even though the Gibbs sampler is geometrically ergodic, many iterations are still recommended even in low dimensions.
These figures suggest empirical diagnostics for the regression parameter $\beta$ as a reasonable choice as opposed to the other parameters $\alpha, \sigma^2$ to determine the reliability of the algorithm in practice.

\begin{figure}[t]
\centering
\begin{subfigure}{.32\textwidth}
  \centering
  \includegraphics[width=\linewidth]{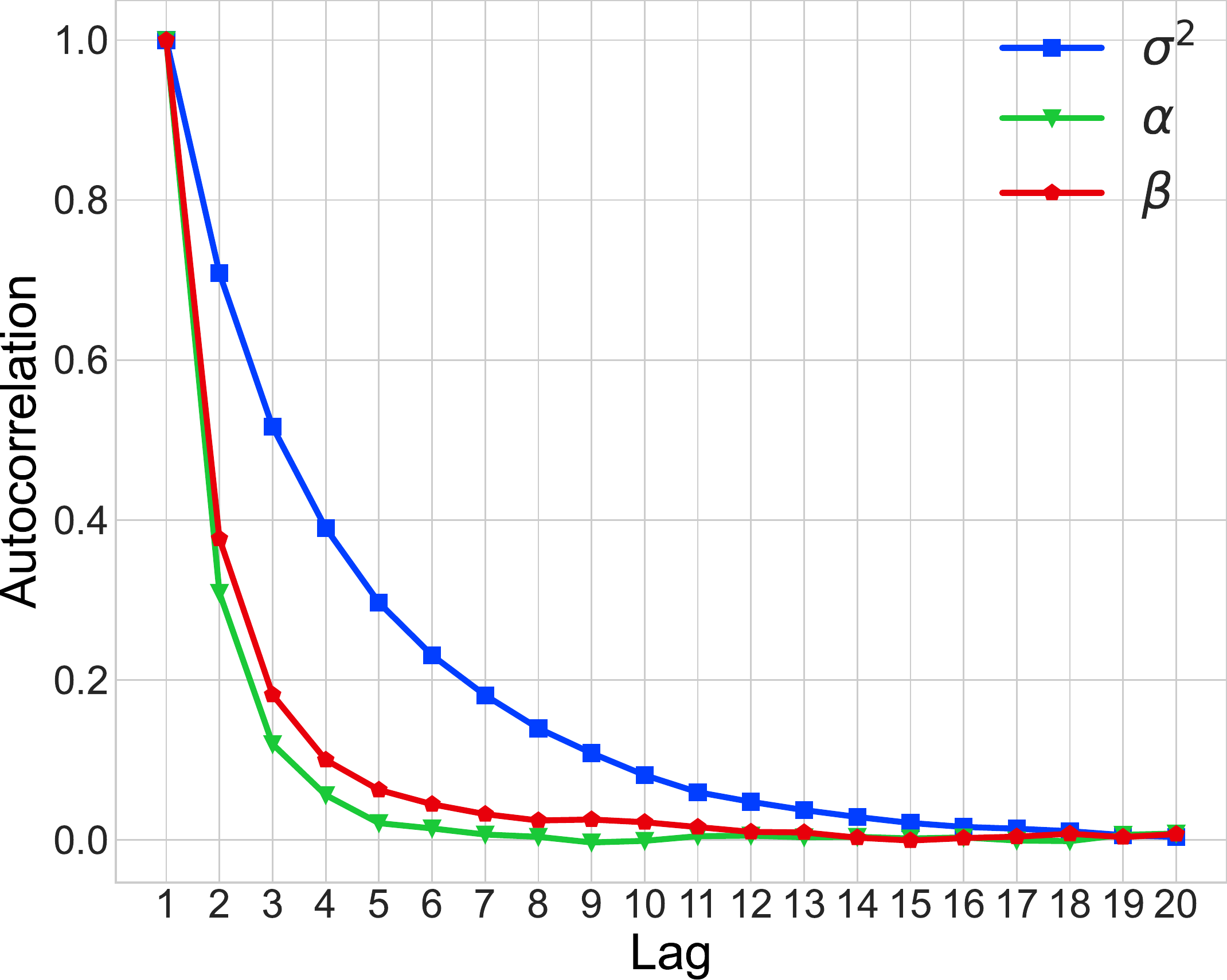}
  \caption{}
\end{subfigure}
\begin{subfigure}{.32\textwidth}
  \centering
  \includegraphics[width=\linewidth]{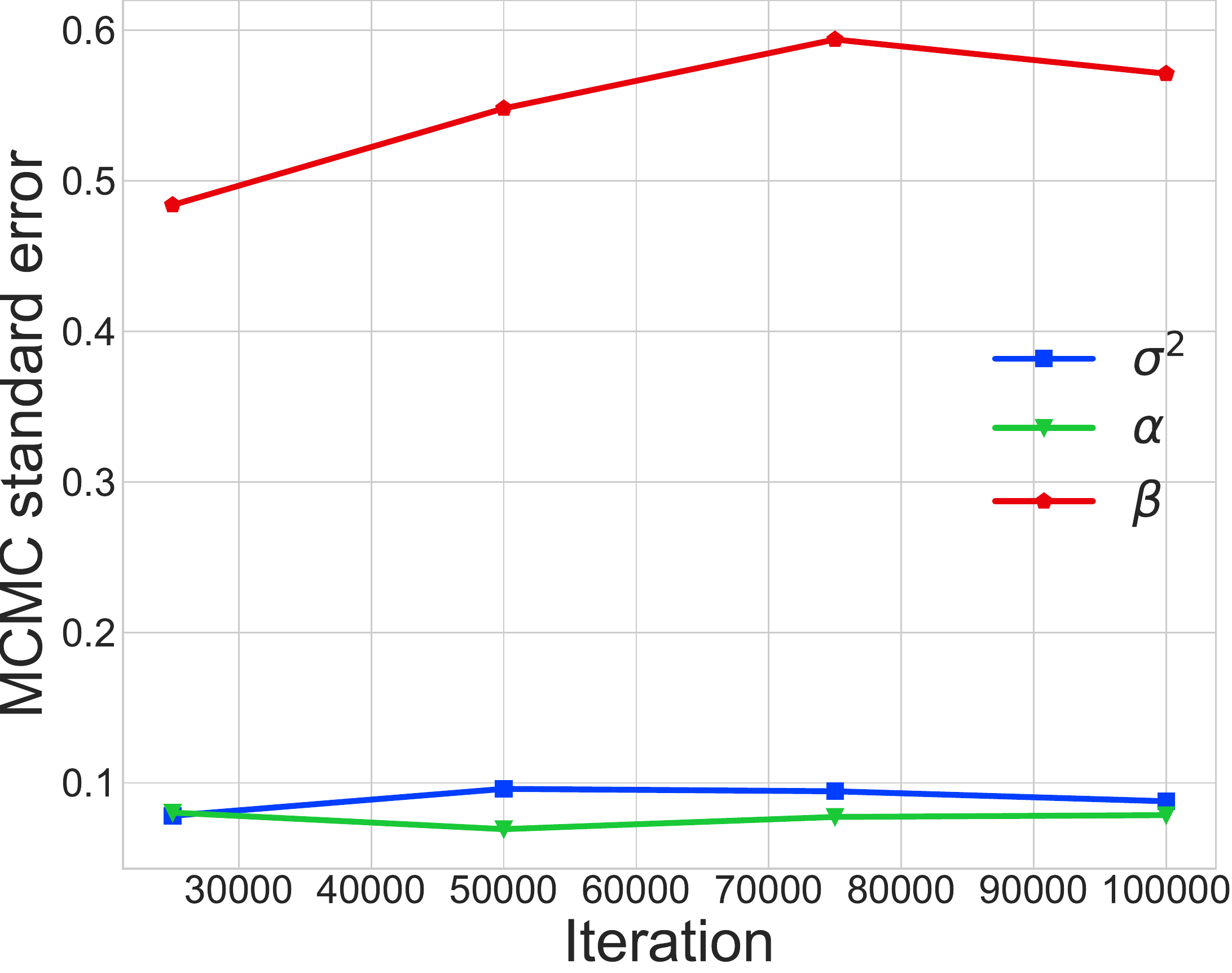}
  \caption{}
\end{subfigure}
\begin{subfigure}{.32\textwidth}
  \centering
  \includegraphics[width=\linewidth]{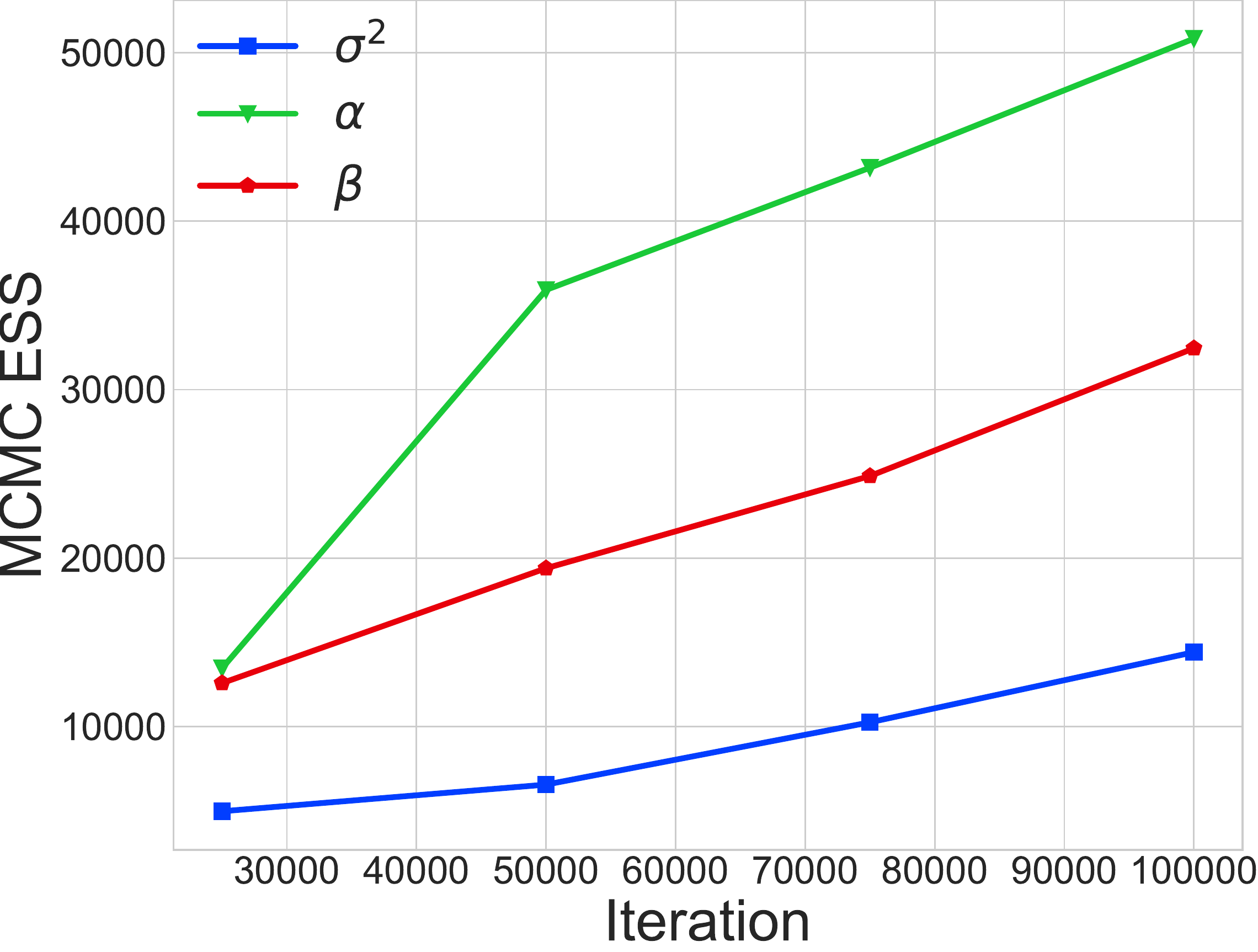}
  \caption{}
\end{subfigure}
\caption{(a) Autocorrelation for each regression parameter (b) MCMC standard errors for the regression parameters (c) MCMC effective sample size plots for each regression parameter} \label{figure:astro_example}
\end{figure}

\section{Conclusion and Future Directions}
\label{section:conclusion_me}

We showed using a $3$-variable deterministic scan Gibbs sampler to sample the posterior in $4$ different multivariate Bayesian EIV regression models with additive Gaussian errors and independent priors is always geometrically ergodic.
This is of pragmatic importance to practitioners as trustworthy estimation from a Gibbs sampler is dependent on the speed of convergence of the Markov chain.
More specifically, time averages from the Markov chains have many practically relevant theoretical guarantees such as a central limit theorem.
Secondly, these Gibbs samplers can be simulated efficiently without the need for complex, intermediate Metropolis-Hastings or rejection sampling steps.
One drawback, however, is our convergence analysis is qualitative as we do not construct an explicit convergence rate.

There are many future research directions in studying the convergence of Gibbs samplers in EIV models. 
It appears reasonable that some Gibbs samplers for generalized linear models such as the Pólya-Gamma sampler will also be geometrically ergodic \citep{Choi2013, Polson2013, Wang2018}.
It seems also interesting to look at alternative errors in the variables such as non-Gaussian or non-additive errors.

\begin{acks}[Acknowledgments]
Thanks to Galin L. Jones for helpful guidance and insightful comments for developing this article.
Thanks to the anonymous referees for recommendations to improve and extend some results from a previous version of this manuscript.
\end{acks}

\bibliographystyle{imsart-nameyear}
\bibliography{main}

\end{document}